\documentclass[11pt]{article}
\usepackage[margin=1in]{geometry}
\usepackage{comment}
\usepackage{graphicx} 
\usepackage{tikz}
\usepackage{authblk}
\usepackage{hyperref}
\usepackage{amsmath,amssymb,amsthm}
\usepackage{cleveref}
\usepackage{thmtools}
\usepackage{xcolor}
\usepackage{todonotes}
\usepackage{float}
\hypersetup{colorlinks=true,linkcolor=blue,filecolor=blue,citecolor=blue,urlcolor=blue}
\declaretheorem[name=Theorem, numberwithin=section]{theorem}
\declaretheorem[name=Lemma, sibling=theorem]{lemma}
\declaretheorem[name=Proposition, sibling=theorem]{proposition}
\declaretheorem[name=Definition, sibling=theorem]{definition}
\declaretheorem[name=Corollary, sibling=theorem]{corollary}
\declaretheorem[name=Conjecture, sibling=theorem]{conjecture}

\declaretheorem[name=Question, sibling=theorem]{question}
\declaretheorem[name=Claim, numbered=yes]{claim}

\declaretheorem[name=Observation, sibling=theorem]{observation}

\DeclareMathOperator{\diam}{diam}

\renewcommand{\hat}[1]{\widehat{#1}}
\newcommand{\say}[1]{``#1''}

\newcommand{\ora}[1]{\overrightarrow{#1}}
\definecolor{dark-green}{rgb}{0.2, 0.5, 0.2}

\renewenvironment{proof}{\par \noindent \textbf{Proof.} }{\hfill$\Box$\medskip}
\newenvironment{proofof}[1]{\par \noindent \textbf{Proof of #1.}
}{\hfill$\lozenge$\medskip}

\title{Inversion diameter and $2$-edge-colored homomorphisms\thanks{AR was supported by Polish National
Science Centre SONATA BIS-12 grant number 2022/46/E/ST6/00143.}}
\author[1]{Carmen Arana}
\author[2]{Thomas Bellitto}
\author[1,3]{Hector Buffière}
\author[4]{Quentin Chuet}
\author[5]{Théo Pierron}
\author[6]{Amadeus Reinald}

\affil[1]{Université Paris Cité, CNRS, IRIF, Paris, France. }
\affil[2]{Sorbonne Université, CNRS, LIP6, Paris, France}
\affil[3]{Centre d'Analyse et de Mathématique Sociales, CNRS, EHESS, Paris, France.}
\affil[4]{LISN, Université Paris-Saclay, Gif-sur-Yvette, France.}
\affil[5]{Univ Lyon, UCBL, CNRS, INSA Lyon, LIRIS, UMR5205, F-69622 Villeurbanne, France.} 
\affil[6]{University of Warsaw, Poland}
\affil[ ]{\small \textit{Email addresses:} \href{mailto:arana@irif.fr}{\tt{arana@irif.fr}},
\href{mailto:thomas.bellitto@lip6.fr}{\tt{thomas.bellitto@lip6.fr}}, \href{mailto:Hector.Buffiere@irif.fr}{\tt{buffiere@irif.fr}},  \href{mailto:quentin.chuet@lisn.fr}{\tt quentin.chuet@lisn.fr},
\href{mailto:theo.pierron@univ-lyon1.fr}{\tt theo.pierron@univ-lyon1.fr}, \href{mailto:reinald@mimuw.edu.pl}{\tt reinald@mimuw.edu.pl}
}

\date{}

\begin{document}

\maketitle
\begin{abstract}
    In an oriented graph, the \emph{inversion} of a subset of vertices $X$ is the operation reversing the direction of every arc with both endpoints in $X$. Given a graph $G$, the \emph{inversion distance} between two orientations $G$ is the minimum number of inversions transforming one into the other.
    The \emph{inversion diameter} $\diam(G)$ is the maximum such distance over all pairs of orientations of $G$.
    Through an equivalent formulation of inversions over $2$-edge-colorings of $G$, we introduce the use of homomorphism-universal 2-edge-colored graphs to obtain bounds on the inversion diameter of various classes of graphs.

    Our first result upper bounds the inversion diameter by a linear function of the acyclic chromatic number, improving on the previous quadratic dependency.
    We then consider the inversion diameter of planar graphs, exhibiting a lower bound of $6$, as well as new lower and upper bounds for those of a given girth, in particular settling the girth $7$ case.
    We then show that any triangle-free graph $G$ with maximum degree $\Delta$ satisfies $\diam(G) \leq \Delta+\log\Delta$, making progress on the conjecture of Havet et al. that $\diam(G) \leq \Delta$.
    Finally, we prove a general result about subdivisions: if a graph has inversion diameter $k$, any of its subdivisions has inversion diameter at most $k+\log k + 5$.
\end{abstract}

\section{Introduction}
Given an oriented graph $\ora G$, and a subset of vertices $X\subseteq V(\ora G)$, the {\em inversion} of $X$ is the operation reversing an arc between two vertices if and only if both belong to $X$.
Inversions were introduced by Belkhechine~\cite{belkhechine2009indecomposabilite} who investigated sequences of inversions turning an oriented graph acyclic.
For this purpose, inversions generalize feedback arc sets, as the deletion of a set of arcs has the same effect as the successive inversions of each pair of endpoints.
The {\em inversion number} of an oriented graph is the minimal number of inversions required to obtain an acyclic orientation.
The study of inversion number was initiated by Belkhechine et al.~\cite{BBBP}, who obtained estimates and structural obstructions for tournaments. It has gained a renewed interest recently, through both algorithmic and asymptotic considerations~\cite{alon, aubian, bang-jensenInversionNumberOriented2022a,duronkArc25,bangjensen2025makingorientedgraphacyclic}.

While the inversion number can be seen as the \say{distance} from an oriented graph towards \textsl{any} acyclic orientation, we are here interested in a finer measure.
Considering two orientations of the same graph, what is the minimum number of inversions required to obtain one from the other?
This question, which can be seen as a reconfiguration problem, was raised recently by Havet et al.~\cite{havet}.
Specifically, given a \textsl{labeled} graph $G$, and two orientations $\ora{G_1},\ora{G_2}$ of $G$, the \emph{(inversion) distance} between $\ora{G_1}$ and $\ora{G_2}$ is defined as the minimum number of inversions required to obtain $\ora{G_2}$ from $\ora{G_1}$.
We stress that all (oriented) graphs here are labeled, meaning that in the above, it is not sufficient to obtain a graph isomorphic to $\ora{G_2}$. Note also that this distance is indeed symmetric, as $\ora{G_1}$ can be obtained from $\ora{G_2}$ by the same sequence used from $\ora{G_1}$ to $\ora{G_2}$.
Then, the {\em inversion diameter}, denoted $\diam(G)$, is defined as the largest inversion distance between any two orientations of $G$.
The authors of~\cite{havet} investigated several structural aspects of inversion diameter, its relation to other parameters, as well as bounds for sparse classes of graphs.

An important result of~\cite{havet} is that the parameter $\diam$ is functionally equivalent to three (already tied) other parameters: the star chromatic number $\chi_s$, the acyclic chromatic number $\chi_a$ and the oriented chromatic number $\chi_o$.
\begin{theorem}[Theorem 1.2 in~\cite{havet}]\label{thm:havet-functional-equiv}
    The parameters $\diam,\chi_a,\chi_s,\chi_o$ are functionally equivalent, and in particular: $\diam \leq \frac{1}{3}(\chi_s^2+\chi_s+1), \frac{2}{3}(\chi_a^2+\chi_a+1), \chi_o^2-1$.
\end{theorem}
The above already tells us that the inversion diameter is bounded on classes where these parameters are, in particular planar~\cite{borodinAcyclicColoringsPlanar1979} and bounded degree~\cite{fertinAcyclicColoringGraphs2005}  graphs.
Still, this gives no indication on the optimality, and it is interesting to obtain bounds specific to $\diam$.
Letting $\Delta(G)$ be the maximum degree of $G$, Havet et al.~\cite{havet} showed that $\diam(G) \leq 2 \Delta(G) -1$, and conjectured that this bound can essentially be halved:
\begin{conjecture}[Conjecture 6.4 in \cite{havet}]\label{conj:delta}
    Every graph $G$ satisfies $\diam(G) \leq \Delta(G)$.
\end{conjecture}
They showed a bound of $\Delta + \lceil \log \Delta \rceil -1$ for bipartite graphs,
but the general case remains one of the main open problems regarding the inversion diameter.
Another approach initiated in~\cite{havet} is to give specific bounds for other sparse classes, such as planar graphs, proper minor-closed classes, and graphs of bounded treewidth.
We will be particularly interested in the first problem, where the authors show a bound of $12$ on the inversion diameter of general planar graphs.
Moreover, they show that this bound can be improved for planar graphs of large girth, for which they give (loose) lower bounds (see also~\Cref{table:bounds}).
A further restriction is the case of outerplanar graphs, which was settled by Wang et al.~\cite{wang2024inversion} with a tight upper bound of $4$.
With respect to other sparse classes, we note that Havet et al. show upper bounds for graphs excluding a fixed minor, and a bound of $2t$ for graphs of treewidth at most $t$, which was shown to be tight by Wang et al.~\cite{wang2024inversion}. 

\paragraph{Contributions and techniques.}
We investigate inversion diameter through several problems, in particular motivated by~\cite{havet}.
Our results are based on inversions of $2$-edge colored graphs, which give rise to an equivalent diameter.
Namely, given a graph $G$ along with a $2$-edge-coloring, the \emph{inversion} of a set of vertices $X$ consists of inverting the color of all edges with both endpoints in $X$ (i.e. all edges of $G[X]$). The resulting diameter, as is noted in~\Cref{ssec:2edgecol}, is equal to $\diam$.
We rely on two tools to give bounds on this diameter and therefore the inversion diameter of various classes.
The first, which we introduce, is the use of homomorphisms-universal $2$-edge-colored graphs for upper bounds on $\diam$. The second tool is a reformulation of $\diam$ in terms of linear algebra, which was observed in~\cite{havet} (\Cref{obs:main}). This allows us to derive both lower and upper bounds.

Our starting observation is that inversions of $2$-edge colored graphs are well-behaved under homomorphisms (\Cref{obs:diam-preserved-homo}).
In particular, one can bound the inversion diameter of a class $\mathcal{C}$, by bounding the diameter of a single \say{universal} graph for $\mathcal{C}$, called a target.
Specifically, a \emph{target} for $\mathcal{C}$ is a fixed $2$-edge-colored graph, onto which every $2$-edge-colored graph in $\mathcal{C}$ admits a homomorphism. Then, any sequence of inversions on the target can be lifted into a sequence for the original graph.

Section~\ref{sec:chia} shows our first application of this framework, by giving a linear upper bound on $\diam$ in terms of $\chi_a$ and $\chi_s$, improving on the quadratic dependency of~\Cref{thm:havet-functional-equiv}.
\begin{restatable}{theorem}{thmacyclic}
\label{thm:acyclic}
    For any graph $G$, $\diam(G) \leq \max(4\chi_a(G)-7,2) \leq \max(4 \chi_s(G) - 7,2)$.
\end{restatable}
\noindent
The second inequality comes from the fact that $\chi_a \leq \chi_s$. We believe these bounds are not tight and could still be improved, at the cost of a more technical proof. However, observe that a linear bound is best possible since cliques of size $k$ have inversion diameter as well as star and acyclic chromatic number $k$.
The main remaining question is whether this linear dependency holds for $\chi_o$ as well, as the upper bound on $\chi_a$ in terms of $\chi_o$ is cubic~\cite{kostochka1997acyclic}.

Then, \Cref{sec:planar} deals with the inversion diameter of planar graphs, and in particular those of a prescribed girth, following~\cite{havet}. There, we improve lower and upper bounds on $\diam$, as summarized in~\Cref{table:bounds}. These results are obtained using both homomorphisms  and the linear algebra expression of $\diam$.

\begin{table}[!ht]
\centering
\begin{tabular}{|c|c|c|}
\hline
girth & lower bound & upper bound \\
\hline
3       &{\color{red} 6}   & 12 \cite{havet}       \\
4       & {\color{red} 4}  & {\color{red} 10}     \\ 
5       &  3               & {\color{red} 7}      \\
6       &   3              & {\color{red} 5}      \\ 
7       &  3               & {\color{red} 3}      \\
8       & 3 \cite{havet}   & 3 \cite{havet}       \\
\hline
\end{tabular}
\caption{Lower and upper bounds on the inversion diameter of planar graphs of a given girth. Previously known bounds are in black, while our improvements are highlighted in red.}\label{table:bounds}
\end{table}

In~\Cref{sec:triangle-free}, we inch closer to~\Cref{conj:delta} for triangle-free graphs, using the algebraic formulation of $\diam$, combined with standard linear algebra. 
\begin{restatable}{theorem}{thmtrianglefree}
\label{thm:main}
    Every triangle-free graph $G$ satisfies $\diam(G)\leq \Delta(G)+\lfloor\log\Delta(G)\rfloor$.
\end{restatable}
\noindent
This asymptotically matches and extends the bound obtained by~\cite{havet} when $G$ is bipartite.

In Section~\ref{sec:subd}, we rely on the linear algebra expression of the inversion diameter to study its behavior with respect to subdivisions.
Namely, we show that subdividing a graph can only increase $\diam$ by a logarithmic term. 
\begin{restatable}{theorem}{thmsubd}
\label{thm:subd}
    Let $H$ be a subdivision of $G$, then $\diam (H) \le \diam (G) + \log(\diam(G))+5$.
\end{restatable}
We have no indication that this result should be tight, and we leave as an open question whether it could be lowered
to $\diam(G) + c$ for some constant $c>0$. 
\begin{question}
    Is there a constant $c$ such that for every graph $G$ and every subdivision $H$ of $G$, then $\diam(H)\leq \diam (G)+c$?
\end{question}

\section{Preliminaries}\label{sec:preliminaries}


We will use the notation $[l,r]:=\{m\in\mathbb{N}\mid l \leq m \leq r\}$, which will be simplified to $[r]$ when $l=1$.
We denote by $\mathbb{F}_2$ the 2-element field, and by $\cdot$ the standard dot-product on $\mathbb{F}_2$-vector spaces.

\subsection{Inversions of 2-edge-colored graphs}\label{ssec:2edgecol}

We will denote a colored graph by $(G,\pi)$, where 
$\pi:E(G)\to\mathbb{F}_2$.
We define the \emph{norm} $\|\pi\|$ of $\pi$ as the minimum length of an inversion sequence turning $(G, \pi)$ into the constant edge-coloring of $G$ in which every edge receives color 0.

While there is no canonical bijection between orientations of $G$ and its $2$-edge colorings, the latter can still be used to capture inversion distances between orientations, and thus $\diam$.
Indeed, consider two orientations $\ora{G_1}$ and $\ora{G_2}$ of a graph $G$. We define the coloration $\pi$ by letting, for every $uv\in E(G)$, $\pi(uv) = 0$ if the orientation of $uv$ is the same in $\ora{G_1}$ and $\ora{G_2}$, and $\pi(uv) = 1$ otherwise.
Then, observe that $\|\pi\|$ is precisely the (oriented) inversion distance between $\ora{G_1}$ and $\ora{G_2}$. In particular, we get:
\begin{observation}
    \label{obs:norm}
    For every graph $G$, $\diam(G)$ is the maximum norm of a $2$-edge-coloring of $G$. 
\end{observation}

\subsection{Linear algebra expression of \texorpdfstring{$\diam$}{diam}}

We now recall the linear algebraic expression of $\diam$ from~\cite{havet}.
Throughout the paper, bold letters will denote vectors, and when assigning vectors to vertices of a graph, vertex $u$ gets assigned vector $\mathbf{u}$. 
Given a $2$-edge coloring $\pi$ of $G$, a \emph{$t$-inversion} for $\pi$ is the assignment of some $\mathbf{u}\in\mathbb{F}_2^t$ to each $u \in V(G)$, such that $\pi(uv)=\mathbf{u\cdot v}$ for every $uv\in E(G)$.
\begin{observation}[Observation 2.2 in~\cite{havet}]
\label{obs:main}
For every graph $G$ and every positive integer $t$, $\diam(G)\leq t$ if and only if every 2-edge-coloring $\pi:E(G)\to\mathbb{F}_2$ admits a $t$-inversion.
\end{observation}

\subsection{Bounding \texorpdfstring{$\diam$}{diam} using homomorphisms}

Here, we show our main use for the $2$-edge-colored formulation of $\diam$, namely its behavior under homomorphisms.
Given two edge-colored graphs $(G,\pi)$ and $(H,\pi')$, we say that $f:V(G)\to V(H)$ is a \emph{homomorphism} if $f(u)f(v)$ is an edge of $H$ whenever $uv$ is an edge of $G$, and moreover, $\pi(uv)=\pi'(f(u)f(v))$.
We write $(G,\pi)\to (H,\pi')$ when there exists such a homomorphism.
The crucial observation then is that norm is non-decreasing under homomorphisms of $2$-edge colored graphs.
Indeed, a $t$-inversion in the image of a homomorphism can be lifted to an inversion in its pre-image, as shown below.

\begin{observation}\label{obs:diam-preserved-homo}
    If $(G,\pi)\to (H,\pi')$, then $\|\pi\|\leq \|\pi'\|$.
\end{observation}

\begin{proof}
    Let $t = \|\pi'\|$ and $\mathbf{i}:V(H) \to \mathbb{F}_2^t$ be a $t$-inversion for $\pi'$. Let $f$ be a homomorphism from $(G,\pi)$ to $(H,\pi')$. We claim that $\mathbf{i}\circ f$ is a $t$-inversion for $\pi$, thus proving that $\|\pi\| \le \|\pi'\|$. Consider any edge $uv \in E(G)$: its endpoints are assigned the vectors $\mathbf{i}(f(u))$ and $\mathbf{i}(f(v))$ respectively, and since $f(u)f(v)$ is an edge in $H$, we have $\mathbf{i}(f(u)) \cdot \mathbf{i}(f(v)) = \pi'(f(u)f(v)) = \pi(uv)$.
\end{proof}

Given a class of graphs $\mathcal{C}$, we say that a $2$-edge-colored graph $(H,\pi)$ is a \emph{target} for $\mathcal{C}$ if every 2-edge-coloring of every element of $\mathcal{C}$ admits a homomorphism to $(H,\pi)$. As an easy consequence of the previous observation, we get the following.

\begin{lemma}
\label{prop:hom}
    If $(H,\pi)$ is a target for $\mathcal{C}$, then $\diam(G)\leq \|\pi\|$ for every $G\in\mathcal{C}$. 
\end{lemma}

It is tempting to generalize the latter proposition to deal with multiple targets. This is indeed possible: say that a set $\mathcal{H}$ of $2$-edge-colored graphs is a \emph{target set} for $\mathcal{C}$ if every $2$-edge-coloring of every element of $\mathcal{C}$ admits a homomorphism to some element of $\mathcal{H}$. \Cref{prop:hom} naturally extends as follows.

\begin{proposition}
    If $\mathcal{H}$ is a target set for $\mathcal{C}$, then $\diam(G)\leq \max_{(H,\pi)\in\mathcal{H}} \|\pi\|$ for every $G\in\mathcal{C}$. 
\end{proposition}

Note however that this is useless in our case, since we are interested in targets with small norm, not with few vertices. In particular, the disjoint union of all graphs in a target set $\mathcal{H}$ for $\mathcal{C}$ is a target for $\mathcal{C}$ of norm exactly $\max_{(H,\pi)\in\mathcal{H}} \|\pi\|$.

\section{Linear bound in the acyclic chromatic number}\label{sec:chia}

This section is devoted to the proof of~\Cref{thm:acyclic}, giving linear bounds on the inversion diameter in terms of the acyclic and star chromatic numbers.
Recall that the \emph{acyclic chromatic number} $\chi_a(G)$ (resp. \emph{star chromatic number} $\chi_s(G)$) of a graph $G$ is defined as the minimum number of colors required in a proper coloring such that the union of any two color classes induces a forest (resp. a star forest).
To show~\Cref{thm:acyclic}, we rely on the 2-edge-coloring reformulation of inversion diameter, as introduced in~\Cref{ssec:2edgecol}.
Then, according to \Cref{prop:hom}, to obtain a bound on the inversion diameter
of graphs with acyclic chromatic number
at most $k$, it suffices to have a bound for a target of this class. That is, a $2$-edge-colored graph onto which every $2$-edge-coloring of a graph $G$ with $\chi_a(G) \leq k$ admits a homomorphism.
It turns out that such targets already exist and are known as the Zielonka graphs~\cite{raspaud1994good,ochem2017homomorphisms}.
\begin{definition}
    The 2-edge-colored Zielonka graph $(SZ_k,\pi_k)$ is a complete $k$-partite graph, whose parts $X_1,\ldots,X_k$ have size $2^{k-1}$. The vertices of $X_i$ are labeled as $(i,\alpha_1, \dots \alpha_k)$ where $1\leq i \leq k$, $\alpha_i=0$, and $\forall j\neq i, \alpha_j \in \{-1,1\}$. For every $i\neq j$, there is an edge between $(i,\alpha_1, \dots \alpha_k)$ and $(j,\beta_1, \dots \beta_k)$ colored $0$ if $\alpha_j\beta_i=1$ and $1$ otherwise (that is, colored
    $\frac{1}{2}(1-\alpha_j\beta_i)$).
\end{definition}

\begin{theorem}[\cite{raspaud1994good}]
Every $2$-edge coloring of a graph $G$ with $\chi_a(G)\leq k$ admits a homomorphism to $(SZ_k,\pi_k)$.
\end{theorem}

We shall prove the following:

\begin{theorem}\label{sz_diam}
    For $k\geq 3$, $\|\pi_k\|\leq 4k-7$. 
\end{theorem}

To prove~\Cref{sz_diam}, we proceed in two steps. First, we invert some parts of $\pi_k$ to create many twins, namely non-adjacent vertices with the same (colored) neighborhoods. The following observation allows us to remove them without changing the norm.

\begin{observation}
\label{lem:twins}
    Let $(G,\pi)$ be a graph, and $v$ a vertex with a twin. Then $\|\pi\|=\|\pi|_{E(G-v)}\|$.
\end{observation}

\begin{proof}
    The inequality $\|\pi|_{E(G-v)}\|\leq \|\pi\|$ is clear. For the converse one, let $H_1, \dots, H_k$ be an inversion sequence rendering $(G-v,\pi|_{E(G-v)})$ colored $0$.
    Let $u$ be the twin of $v$.
    Construct sets $J_1, \dots, J_k$ such that $J_i=H_i \cup\{v\} $ if $u \in H_i$,
    and $J_i=H_i$ otherwise. Now $J_1,\ldots, J_k$ is an inversion sequence ending with $(G,\pi)$ colored $0$, which concludes. 
\end{proof}

Using this observation, we build a subgraph $(G_k,\sigma_k)$ of $(SZ_k,\pi_k)$ such that $\| \pi_k\|\leq \|\sigma_k\|+2k-2$. We then conclude the proof of Theorem~\ref{sz_diam} by showing that $\|\sigma_k\|\leq 2k-5$.

Let $(G_k,\sigma_k)$ the subgraph of $(SZ_k,\pi_k)$ induced by all
the vertices $(i,\alpha)$ such that $\alpha_j = 1$ for all $j \in [i+1]\setminus \{i\}$. Theorem~\ref{sz_diam} is then a consequence of the following.

\begin{lemma}
\label{lem:1}
    For every $k>0$, in $2k-2$ inversions, one can transform $\pi_k$ into a coloring where the subgraph $G_k$ is colored with $\sigma_k$, and all remaining vertices have a twin in $(G_k,\sigma_k)$.
\end{lemma}

\begin{proof}
    We proceed by induction on $k$. If $k=1$, then $(SZ_k,\pi_k)$ and $(G_k,\sigma_k)$ are the $1$-vertex graph, and the inequality holds. Assume now $k>1$. 

    Given a vertex $u=(i,\alpha_1,\ldots,\alpha_k)$ in $SZ_k$, let us denote by $-u$ the vertex $(i,-\alpha_1,\ldots,-\alpha_k)$. Observe that by construction $u$ and $-u$ are anti-twins, that is, for every vertex $v$, either $v$ is adjacent to none of $u$ and $-u$, or it is adjacent to both, with edges of different colors.

    To prove the induction step, we first note that we may as well consider
    the opposite coloring $(SZ_k,1-\pi_k)$.
    Indeed, $(SZ_k,1-\pi_k)$ and $(SZ_k,\pi_k)$ are isomorphic, witnessed by the
    map  $\varphi : (i,\alpha_1,\ldots,\alpha_k) \mapsto (i,\alpha_1,\ldots,\alpha_i,-\alpha_{i+1},\ldots,-\alpha_k)$, as for $i<j$, $1- \frac{1}{2}(1-\alpha_j\beta_i) = \frac{1}{2}(1-(-\alpha_j)\beta_i)$.
    
    Let us partition the vertices of $SZ_k$ into two parts $A,B$ with
    \[
    A := \{ (1,\alpha); \alpha_2 = -1\}\cup\{ (i,{\alpha}); i\neq 1 \text{ and } \alpha_1=1\}.
    \] Observe that for every $u\in A$, $-u\in B$. Let us invert $A$ then $B$ in $(SZ_k,1-\pi_k)$, obtaining a new coloring $(SZ_k, \rho)$. Each vertex $u\in A$ is a twin of $-u$ in $(SZ_k,\rho)$, since for every $v\in A$ (resp. $v\in B$), $uv$ (resp. $-uv$) changed color but not $-uv$ (resp. $uv$). We may then forget about the vertices in $B$ by Lemma~\ref{lem:twins}.
    
    Let $A^+$ be the set of vertices with $\alpha_1=1$. In $(SZ_k,1-\pi_k)$, vertices of $A^+$ induce a copy of $(SZ_{k-1},1-\pi_{k-1})$, hence they induce a copy of $(SZ_{k-1},\pi_{k-1})$ in $(SZ_k, \rho)$. By induction, one can get a coloring where $G_{k-1}$ is colored with $\sigma_{k-1}$ and all remaining vertices have a twin in $(G_k,\sigma_k)$ in $2k-4$ inversions. Note that if two vertices of $A^+$ have the same colored neighborhood in $A^+$, then they are twins in $(SZ_k, \rho)$. Indeed, they must lie in the same $X_i$, and moreover if $u=(1,{\alpha})\in A\setminus A^+$, the edges linking them to $u$ have the same color (depending only on $\alpha_i$).
    
    Moreover, adding the vertices of $A\setminus A^+$ to the subgraph $(G_{k-1},\sigma_{k-1})$ yields precisely the graph $(G_k,\sigma_k)$ (since $A$ has been inverted). This concludes the induction.
\end{proof}

\begin{lemma}
    \label{lem:2}
    For $k\geq 3$, $\|\sigma_k\|\leq 2k-5$.
\end{lemma}

\begin{proof}
    We proceed by induction on $k$. The case $k=3$ follows from $\sigma_3$ having a single edge labeled $0$.
    Assume $k>3$. Consider the vertex $u$ of $G_k$ in $X_k$, and let $Y$ its neighbors $v$ such that $uv$ has color $1$. Invert $Y\cup\{u\}$ then $Y$. This ensures that all edges incident to $u$ are now colored $0$, and the others did not change. After these two inversions, each vertex $(i,\alpha_1,\ldots,\alpha_k)$ with $i<k$ is now a twin of $(i,\alpha_1,\ldots,\alpha_{k-1},-\alpha_k)$. Deleting these twins yields a copy of $(G_{k-1},\sigma_{k-1})$, hence by Lemma~\ref{lem:twins}, $\|\sigma_k\|\leq 2+\|\sigma_{k-1}\|\leq 2k-5$ using induction.
\end{proof}

This concludes the proof of \Cref{sz_diam}.
Then, applying \Cref{prop:hom} to \Cref{sz_diam} achieves to bound $\diam$ in terms of the acyclic (and star) chromatic number.
\thmacyclic*

\section{Planar graphs}
\label{sec:planar}

This section is devoted to our results on planar graphs, summarized in \Cref{table:bounds}.
We begin by providing two constructions, witnessing a lower bound of $6$ for the diameter of planar graphs in~\Cref{ssec:lb-planar-6}, and $4$ for that of planar graphs with girth $4$ in~\Cref{ssec:lb-planar-girth-4}.
Then, we present several upper bounds.
These bounds rely on two kinds of arguments: standard discharging for girths 4 and 7, and another use of the homomorphism framework (\Cref{prop:hom}) for girths 5 and 6.

Throughout this section, we will use the following conventions.
In our figures, we represent edges colored $0$ with a solid black line, and those colored $1$ with a red dotted line.
Given an integer $d$, a vertex $v$ in a graph is respectively a $d$-vertex, a $d^+$ vertex, or a $d^-$-vertex if it has degree $d$, at least $d$, or at most $d$ respectively. We extend this terminology to define $d$-faces, $d^+$-faces and $d^-$-faces according to their degree analogously.

\subsection{General lower bound of 6}\label{ssec:lb-planar-6}

Here we construct a planar graph with inversion diameter at least 6, improving the previously known lower bound of 5. We use the framework given by Observation~\ref{obs:main}, hence our goal is actually to provide a $2$-edge-coloring and to show that no $5$-inversion exists. The construction relies on several gadgets restricting the available 5-dimensional vectors on some vertices in every $5$-inversion. Each gadget is then glued to the vertices of the next one, until no more vectors are available, which ensures that no 5-inversion exists, and thus that the inversion diameter of the resulting graph is at least 6. 

The first gadget $G_1$ is an edge colored 1, which forbids the null vector (having all coordinates equal to 0) on both its endpoints. Therefore, up to adding a copy of $G_1$ (that is, a pendant edge) to each upcoming vertex, one may assume that the null vector cannot be used anywhere.

The second gadget $G_2$ consists of a path of length two with distinct colors, which prevents its endpoints from being assigned the same vector. This is summarized in the following observation. 

\begin{observation}
    \label{obs:path}
    If $G$ contains two edges $uv$ and $uw$ of different colors, then $v$ and $w$ cannot receive the same vector.
\end{observation}

\begin{proof}
    We must have $\mathbf{u} \cdot \mathbf{w}\neq \mathbf{v} \cdot \mathbf{w}$, so $\mathbf{v}-\mathbf{u} \neq 0$.
\end{proof}

Therefore, for every edge $uv$ in the graph, adding a copy of $G_2$ (that is a vertex $w$ and edges $uw,vw$ with distinct colors) preserves planarity, and guarantees that $u$ and $v$ cannot get the same vector. We will apply this operation to every edge of the upcoming graphs, which forbids using the same vector on adjacent vertices.

The third gadget, $G_3$, is the graph constructed from the graph depicted in Figure~\ref{gadget5} by gluing the gadget $G_2$ on each edge, and the gadget $G_1$ on each vertex. 
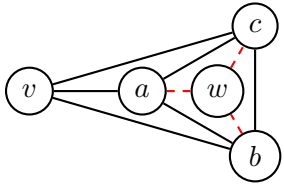
\begin{figure}[H]
    \centering
    \begin{tikzpicture}[thick, every node/.style={draw,circle}]
    \node (v) at (-.5,0) {$v$};
    \node (v1) at (1,0) {$a$};
    \node (v2) at (2,0) {$w$};
    \node (v3) at (2.5,-0.866) {$b$};
    \node (v4) at (2.5,.866) {$c$};
    \draw (v) -- (v1) -- (v3) -- (v) -- (v4) -- (v1);
    \draw (v3) -- (v4);
    \draw[dashed,red] (v3) -- (v2) -- (v4);
    \draw[dashed,red] (v2) -- (v1);
    \end{tikzpicture}
    \caption{A planar graph with a distinguished vertex $v$ used to construct gadget $G_3$}
    \label{gadget5}
\end{figure}

\begin{lemma}
     In any $5$-inversion of $G_3$, the vertex $v$ cannot be inverted at each step, i.e. labeled $(1,1,1,1,1)$. 
\end{lemma}

\begin{proof}
    If $\mathbf{v}=(1,1,1,1,1)$, then vectors $\mathbf{a},\mathbf{b},\mathbf{c}$ must be pairwise orthogonal (as vertices the $a,b,c$ induce a triangle) and of even weight (due to their adjacency to $v$). Moreover, the glued copies of $G_1$ and $G_2$ force them to be non-null and pairwise distinct. This ensures (up to permutation of coordinates and vertices) that $\mathbf{a}=(1,1,1,1,0)$, $\mathbf{b}=(1,1,0,0,0)$ and $\mathbf{c}=(0,0,1,1,0)$, and in particular $\mathbf{a}+\mathbf{b}+\mathbf{c}=\mathbf{0}$. Therefore we obtain $0=(\mathbf{a}+\mathbf{b}+\mathbf{c})\cdot \mathbf{w}=1+1+1=1$, which is absurd.
\end{proof}

We now introduce a fourth and final gadget $G_4$, constructed from the graph depicted in Figure~\ref{gadget4}, by identifying each of its vertices with the distinguished vertex $v$ of a copy of $G_3$, and by adding a copy of $G_2$ on each edge.

\begin{lemma}
     In any $5$-inversion of $G_4$, the vertex $v$ is inverted an odd number of times.
\end{lemma}

\begin{figure}[H]
    \centering
    
    \begin{tikzpicture}[thick, every node/.style={draw,circle}]
    \node (0) at (0,0) {$v$};
    \node (1) at (-135:3) {};
    \node (3) at (-90:2.5) {};
    \node (2) at (-45:3) {};
    \node (4) at (-112.5:2.25) {};
    \node (7) at (-67.5:2.25) {};
    \node (5) at (-123.75:1.5) {};
    \node (6) at (-100.75:1.5) {};
    \node (9) at (-79.25:1.5) {};
    \node (8) at (-55.75:1.5) {};
    \node (10) at (-90:3.5) {};
    \draw (0) -- (1) -- (4) -- (0) -- (3) -- (7) --(2) -- (0) -- (7);
    \draw (4) -- (3) -- (7);
    \draw (2) --(3) --(1);
    \draw[red, dashed] (0) --(5) --(1) -- (10) -- (2) -- (8) -- (0) -- (9) -- (3) -- (6) -- (0);
    \draw[red,dashed] (0) to (2.5,0) to (2.5,-3.5) to (10);
    \draw[red,dashed,bend left] (2) to (1);
    \draw[red,dashed] (8) --(7) -- (9);
    \draw[red,dashed] (6) --(4) -- (5);
    \end{tikzpicture}
    \caption{A planar graph with a distinguished vertex $v$ used to construct the gadget $G_4$}
    \label{gadget4}
\end{figure}
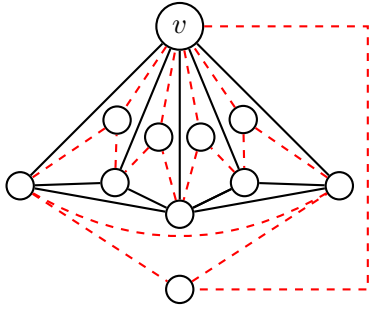

This lemma can be proven by case analysis, using the Python code available at \url{https://github.com/quentinchuet/Inversion-diameter}. In particular, we may now assume that only vectors with 1 or 3 coordinates equal to 1 are allowed in the following graphs. We may now construct a final graph $G$ from the graph depicted in Figure~\ref{gadget13} by identifying each of its vertices with the distinguished vertex of a copy of $G_4$, and gluing $G_2$ on each edge. The following lemma concludes the proof, by showing that $G$ cannot be $5$-inverted using only these vectors. 

\begin{lemma}
    There is no $5$-inversion of $G$.
\end{lemma}

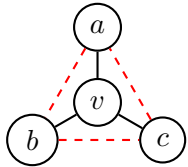
\begin{figure}[H]
    \centering
    \begin{tikzpicture}[thick, every node/.style={draw,circle}]
    \node (v) at (0,0) {$v$};
    \node (a) at (90:1) {$a$};
    \node (b) at (210:1) {$b$};
    \node (c) at (330:1) {$c$};
    \draw (a) -- (v) -- (b);
    \draw (v) -- (c);
    \draw[red,dashed] (a) -- (b) -- (c) -- (a);
    \end{tikzpicture}
    \caption{A planar graph that does not admit a solution with only vectors of weight $1$ or $3$}
    \label{gadget13}
\end{figure}

\begin{proof}
    Due to the glued copies of $G_4$, the vertices $a,b,c$ and $v$ must be inverted exactly $1$ or $3$ times. Therefore, up to permutation of coordinates and vertices, we must have $\mathbf{a}=(1,0,0,0,0)$, $\mathbf{b}=(1,1,1,0,0)$ and $\mathbf{c}=(1,0,0,1,1)$, hence $\mathbf{a} + \mathbf{b} + \mathbf{c} = (1,1,1,1,1)$. But since $\mathbf{v}$ must have an odd weight, we conclude that $1 = (\mathbf{a} + \mathbf{b} + \mathbf{c}) \cdot \mathbf{v} = 0 + 0 + 0 = 0$, which is absurd.
\end{proof}

\subsection{Lower bound for girth 4}\label{ssec:lb-planar-girth-4}

We now proceed with studying planar graphs of girth 4, and start with providing a graph of inversion diameter at least 4, as summarized in the following proposition.

\begin{proposition}
The edge-colored graph from \Cref{girth4} admits no 3-inversion. 
\end{proposition}

\begin{figure}[ht]
    \centering
    \begin{tikzpicture}[thick, every node/.style={draw,circle},xscale=.75]
    \node (0) at (0,0) {};
    \node (1) at (1,0) {};
    \node (2) at (2,0) {};
    \node (3) at (4,0) {};
    \node (4) at (5,0) {};
    \node (5) at (6,0) {};
    \node (6) at (8,0) {};
    \node (7) at (9,0) {};
    \node (8) at (10,0) {};
    \node (9) at (12,0) {};
    \node (10) at (13,0) {};
    \node (11) at (14,0) {};
    \node (-1) at (0,1) {};
    \node (a) at (5.5,2) {};
    \node (b) at (5.5,-2) {};
    \draw (0) -- (a) -- (2) -- (b) -- (0) -- (1);
    \draw (4) -- (3) -- (a) -- (5);
    \draw (7) -- (6) -- (b) -- (8);
    \draw (10) -- (9);
    \draw[red,dashed] (-1) -- (0);
    \draw[red,dashed] (1) -- (2);
    \draw[red,dashed] (4) -- (5) -- (b) -- (3);
    \draw[red,dashed] (7) -- (8) -- (a) -- (6);
    \draw[red,dashed] (10) -- (11) -- (a) -- (9) -- (b) -- (11);
    \end{tikzpicture}
    \caption{A girth 4 planar graph with inversion diameter 4, edges colored 1 are dotted.}
    \label{girth4}
\end{figure}
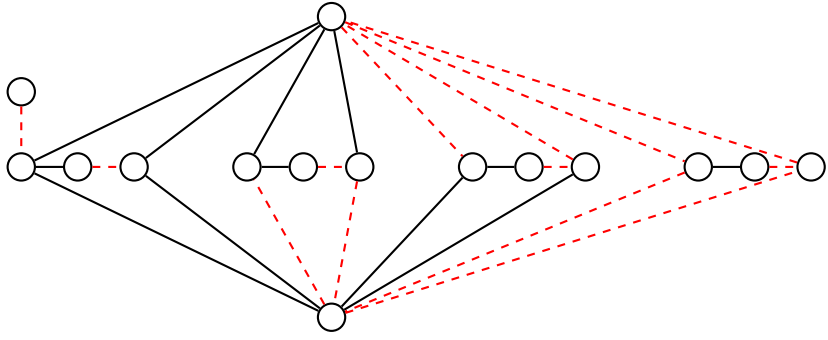

\begin{proof}
Observe that any two vertices of degree $3$ are adjacent with different colors to a third vertex, hence they must receive pairwise distinct vectors by \Cref{obs:path}. Moreover, they are all incident with an edge colored $1$, ensuring that they cannot be labeled with the null vector. Therefore, any vector assignment must use at least 9 vectors, which in turn ensures that their dimension is at least 4.
\end{proof}

\subsection{Upper bound for girth 4}

In this section, we show that every planar graph of girth 4 has inversion diameter at most 10. We use the discharging method to  prove a stronger statement, where each vertex is inverted exactly thrice, summarized in Theorem~\ref{thm:girth4-upper}.

Given an edge-colored graph $(G,\pi)$, and $0 \leq k \leq d$, a $\binom{d}{k}$-vector is a vector in $\mathbb{F}_2^d$ with exactly $k$ non-zero entries, and a $\binom{d}{k}$-inversion is a $d$-inversion where the vector associated with each vertex is a $\binom{d}{k}$-vector.
\begin{theorem}\label{thm:girth4-upper}
    Let $(G,\pi)$ be a triangle-free 2-edge-colored planar graph. Then $(G,\pi)$ admits a $\binom{10}{3}$-inversion.
\end{theorem}

Let us start with a few technical results whose proofs rely heavily on the following lemma, proven by case analysis using the Python code available at \url{https://github.com/quentinchuet/Inversion-diameter}.

\begin{lemma}\label{lem:deg3choice}
    Let $k\in [3]$, $\mathbf{u_1},\ldots,\mathbf{u_k}$ a family of pairwise distinct $\binom{10}{3}$-vectors and $a_1,\ldots,a_k \in \mathbb{F}_2$. Then there exist at least 7 distinct $\binom{10}{3}$-vectors $\mathbf{v_1}, \dots, \mathbf{v_7}$ such that $\mathbf{u_i \cdot v_j} = a_i$ for all $i\leq k$ and $j \in \{1,\dots,7\}$.
\end{lemma}

\begin{corollary}
\label{cor:deg3choice}
    If $(G,\pi)$ admits a $\binom{10}{3}$-inversion $\phi$, and $v$ is a vertex of degree at most $3$, then there exist at least 7 distinct $\binom{10}{3}$-vectors $\mathbf{v_1}, \dots, \mathbf{v_7}$ such that $\phi(v) \gets \mathbf{v_i}$ remains a $\binom{10}{3}$-inversion.
\end{corollary}
\begin{proof}
    Let $u_1,\ldots,u_k$ ($k\leq 3$) be the neighbors of $v$.
    We apply \Cref{lem:deg3choice} to the set $\{\phi(u_1),\ldots,\phi(u_k)\}$, with the values $a_i=\pi(vu_i)$ to get seven $\binom{10}{3}$-vectors.
    
    Note that if $\phi(u_i)=\phi(u_j)$ then $\pi(vu_i)=\pi(vu_j)$. Therefore, assigning any of the obtained vectors to $v$ yields a $\binom{10}{3}$-inversion.
\end{proof}

We will now prove Theorem~\ref{thm:girth4-upper}.
Suppose the theorem does not hold, and throughout the remainder of the subsection, let $G$ be a counterexample which first minimizes the number of $3^+$-vertices, then minimizes the total number of vertices.
Let us fix a planar embedding of $G$.
We begin by showing several claims about the structure of $G$, which will produce a contradiction. 

First we show that $G$ has no cut-vertex; as a consequence, $G$ has no $1$-vertex, and no boundary of a face $G$ contains the same vertex twice.
\begin{claim}\label{claim:cut}
    $G$ has no cut-vertex.
\end{claim}

\begin{proof}
If $G$ has a cut-vertex $v$, by minimality, each of its 2-connected components admits a valid $\binom{10}{3}$-inversion. Since all the vectors assigned to $v$ in each $2$-connected components are $\binom{10}{3}$-vectors, it suffices to permute coordinates in each component to make all vectors assigned to $v$ identical. This yields a $\binom{10}{3}$-inversion of $(G,\pi)$, contradicting the minimality of $G$.
\end{proof}

\begin{claim}\label{claim:2-4}
$G$ has no $2$-vertex adjacent to a $4^-$-vertex.
\end{claim}
\begin{proof}
    Suppose $u$ is a $2$-vertex adjacent to $v$, a $4^-$-vertex, and another vertex $w$. Consider a $\binom{10}{3}$-inversion of $(G-u,\pi|_{E(G-u)})$ obtained by minimality. Since $v$ has degree at most $3$ in $G-u$, we may choose $\mathbf{v}$ distinct from $\mathbf{w}$ by Corollary~\ref{cor:deg3choice}. We can then extend the $\binom{10}{3}$-inversion to $u$ by Lemma~\ref{lem:deg3choice}, a contradiction. 
\end{proof}

\begin{claim}\label{claim:no-4cycle-2-2-vertices}
$G$ has no $4$-cycle with two $2$-vertices.
\end{claim}
\begin{proof}
Suppose $G$ has two $2$-vertices $u,v$ in some $4$-cycle. By the previous claim, $u$ and $v$ cannot be adjacent. Let $x$ and $y$ be the two common neighbors of $u$ and $v$. If $\pi(ux) = \pi(uy)$, we can easily extend a $\binom{10}{3}$-inversion from $G-u$ obtained by minimality using Lemma~\ref{lem:deg3choice}. Otherwise, $\pi(ux) \ne \pi(uy)$ and we consider $G-v$, in which every $\binom{10}{3}$-inversion assigns different vectors to $x$ and $y$. We may thus apply Lemma~\ref{lem:deg3choice} again to obtain a vector for $v$, yielding a $\binom{10}{3}$-inversion of $(G,\pi)$, a contradiction.
\end{proof}

\begin{claim}\label{claim:3vertex}
$G$ has no $3$-vertex.
\end{claim}
\begin{proof}
Suppose $v$ has exactly three neighbors $u_1, u_2, u_3$ in $G$ (which themselves must be $3^+$-vertices, by \Cref{claim:cut} and \Cref{claim:2-4}).
Consider the graph $(G',\pi')$ constructed from $(G,\pi)$ by deleting $v$ and for each $1\leq i<j\leq 3$, adding a common neighbor to $u_i$ and $u_j$ with edges of different colors (see \Cref{fig:transfo_deg3}).

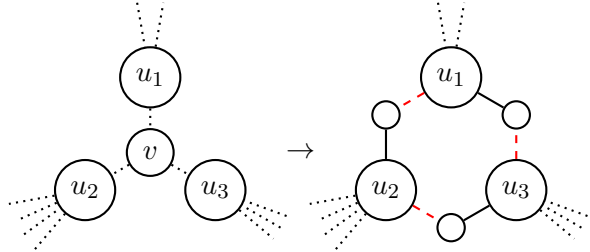
\begin{figure}[!ht]
\centering
    \begin{tikzpicture}[thick, v/.style={draw,circle}]
    \node[v] (v) at (0,0) {$v$};
    \node[v] (a) at (90:1) {$u_1$};
    \node[v] (b) at (210:1) {$u_2$};
    \node[v] (c) at (330:1) {$u_3$};
    \draw[dotted] (a) -- (v) -- (b);
    \draw[dotted] (v) -- (c);
    \draw[dotted] (95:2) -- (a) -- (85:2);
    \draw[dotted] (213:2) -- (b) -- (207:2);
    \draw[dotted] (220:2) -- (b) -- (200:2);
    \draw[dotted] (335:2) -- (c) -- (325:2);
    \draw[dotted] (330:2) -- (c);
    \node at (2,0) {$\to$};
    \tikzset{xshift=4cm}
    \node[v] (a) at (90:1) {$u_1$};
    \node[v] (b) at (210:1) {$u_2$};
    \node[v] (c) at (330:1) {$u_3$};
    \node[v] (ab) at (150:1) {};
    \node[v] (bc) at (-90:1) {};
    \node[v] (ac) at (30:1) {};
    
    \draw[dotted] (95:2) -- (a) -- (85:2);
    \draw[dotted] (213:2) -- (b) -- (207:2);
    \draw[dotted] (220:2) -- (b) -- (200:2);
    \draw[dotted] (335:2) -- (c) -- (325:2);
    \draw[dotted] (330:2) -- (c);
    
    \draw[red, dashed] (a) -- (ab);
    \draw[red, dashed] (b) -- (bc);
    \draw[red, dashed] (c) -- (ac);
    \draw (ab) -- (b);
    \draw (bc) -- (c);
    \draw (ac) -- (a);
    \end{tikzpicture}
\caption{Transformation of degree-3 vertices.}
\label{fig:transfo_deg3}
\end{figure}

$G'$ is still a triangle-free planar graph, and has strictly less $3^+$-vertices than $G$, therefore $(G',\pi')$ admits a $\binom{10}{3}$-inversion by minimality. Moreover, $u_1$, $u_2$, $u_3$ must receive distinct vectors by construction (having a common neighbor with different colors). We can then extend the $\binom{10}{3}$-inversion to $v$ by Lemma~\ref{lem:deg3choice}, a contradiction.
\end{proof}

\begin{claim}\label{claim:567}
    If $\deg(v) = 4 + t$ for some $t \in \{1,2,3\}$, then $v$ is adjacent to at most $t$ vertices of degree $2$.
\end{claim}
\begin{proof}
    If $v$ has $t+1$ neighbors $u_1, \dots, u_{t+1}$ of degree $2$, let $w_1,\dots,w_{t+1}$ be their respective (unique) other neighbors, and remove $u_1,\dots,u_{t+1}$ such that $v$ has $3$ neighbors left. We obtain a $\binom{10}{3}$-inversion of the resulting graph by minimality, and use \Cref{cor:deg3choice} to ensure that $\mathbf{v}$ is distinct from $\mathbf{w_1},\ldots,\mathbf{w_{t+1}}$. We can then extend the inversion to $u_{1},\dots,u_{t+1}$ by \Cref{lem:deg3choice}, a contradiction.
\end{proof}

Given that we have fixed an embedding of $G$, when considering a specific vertex $v$ we have a cyclic order $\sigma_v$ on its neighbors. If $u \in N(v)$ has degree $2$, we shall say that $u$ is (see \Cref{fig:iso_fringe_protected}):
\begin{itemize}
    \item an \emph{isolated neighbor} of $v$ if both its predecessor and its successor in $\sigma_v$ have degree $\ge 3$.
    \item a \emph{protected neighbor} of $v$ if both its predecessor and its successor in $\sigma_v$ have degree $2$.
    \item a \emph{fringe neighbor} of $v$ otherwise, i.e. exactly one of its successor or predecessor in $\sigma_v$ has degree $2$.
\end{itemize}

\begin{figure}[!ht]
\centering
    \begin{tikzpicture}[thick, v/.style={draw,circle,inner sep=2pt}]
    \node[v] (v) at (0,0) {$v$};
    \node[v,label=above:{\scriptsize isolated}] (2) at (0,1) {$2$};
    \node[v] (g) at (-1,1.5) {$3^+$};
    \node[v] (d) at (1,1.5) {$3^+$};
    \draw (g) -- (v) -- (d);
    \draw (v) -- (2);
    \tikzset{xshift=4cm}
    \node[v] (v) at (0,0) {$v$};
    \node[v,label=above:{\scriptsize fringe}] (2g) at (-.5,1) {$2$};
    \node[v,label=above:{\scriptsize fringe}] (2d) at (.5,1) {$2$};
    \node[v] (g) at (-1.5,1.5) {$3^+$};
    \node[v] (d) at (1.5,1.5) {$3^+$};
    \draw[bend right] (g) to (v);
    \draw[bend right] (v) to (d);
    \draw (2d) -- (v) -- (2g);
    \tikzset{xshift=4cm}
    \node[v] (v) at (0,0) {$v$};
    \node[v,label=above:{\scriptsize fringe}] (2g) at (-.75,1) {$2$};
    \node[v,label=above:{\scriptsize fringe}] (2d) at (.75,1) {$2$};
    \node[v,pin=above:{\scriptsize protected}] (2m) at (0,1) {$2$};
    \node[v] (g) at (-1.5,1.5) {$3^+$};
    \node[v] (d) at (1.5,1.5) {$3^+$};
    \draw[bend right] (g) to (v);
    \draw[bend right] (v) to (d);
    \draw (2d) -- (v) -- (2g);
    \draw (2m) -- (v);
    \end{tikzpicture}
\caption{Examples of isolated, fringe, and protected neighbors of $v$}
\label{fig:iso_fringe_protected}
\end{figure}
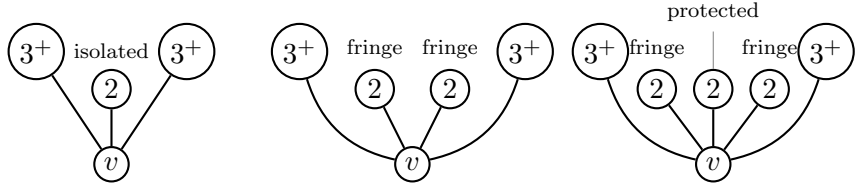

Note that if all neighbors of $v$ have degree $2$, then they are all protected. 
Moreover, the fact that $G$ is triangle-free, along with \Cref{claim:no-4cycle-2-2-vertices} ensure the following.

\begin{claim}\label{claim:fringe_protected}
    A fringe neighbor belongs to at least one $5^+$-face (containing the other $2$-vertex). Moreover, a protected neighbor is adjacent to at least two $5^+$-faces. 
\end{claim}

We now derive a contradiction using a simple discharging argument, which amounts to double-counting the sum $\omega(G):= \sum_v \omega(v) + \sum_f \omega(f)$ of the following initial charges.

\noindent\textbf{Initial charges:}
\begin{itemize}
    \item $\omega(v) = \deg(v) - 4$.
    \item $\omega(f) = \deg(f) - 4$.
\end{itemize}

By Euler's formula, $w(G) =  -8$. We now redistribute the charges according to the following rules, and reach a contradiction by proving that every vertex and every face has a non-negative charge after discharging. This will conclude and yield~\Cref{thm:girth4-upper}.

\noindent\textbf{Discharging rules:}
\begin{itemize}
    \item Every $5^+$ vertex gives $1$ to its isolated neighbors, $\frac{3}{4}$ to its fringe neighbors, and $\frac{1}{2}$ to its protected neighbors.
    \item Every $5^+$-face $f$ gives $\frac{1}{2}$ to its vertices of degree $2$.
\end{itemize}

\begin{claim}
    After discharging, all faces have non-negative charge.
\end{claim}

\begin{proof}
    There are no $3$-faces due to the absence of triangles, and $4$-faces are unaffected. Faces of degree $d \ge 5$ contain at most $\lfloor\frac{d}{2}\rfloor$ vertices of degree $2$ by \Cref{claim:2-4}, hence they end up with a charge of at least $d-4-\lfloor d/2\rfloor/2 \geq 0 $.
\end{proof}

\begin{claim}
    After discharging, all $2$-vertices have non-negative charge.
\end{claim}

\begin{proof}
    Let $v$ be a vertex of degree $2$ (which initially has a charge of $-2$). Let $x,y$ be the two neighbors of $v$, note that these are $5^+$-vertices by \Cref{claim:2-4}. If $v$ is a protected neighbor of $x$ or a protected neighbor of $y$, we are done. Indeed, $v$ is then incident to two $5^+$-faces by \Cref{claim:fringe_protected}, receiving charge at least $1$ from them, and receives at least $1$ in total from $x$ and $y$. 
    
    If $v$ is an isolated neighbor of $x$, then either $v$ is an isolated neighbor of $y$, in which case we are done, or $v$ belongs to at least one $5^+$-face, and we are also done. By symmetry, if $v$ is an isolated neighbor of $y$, we are done.

    Finally, if $v$ is both a fringe neighbor of $x$ and $y$, then $v$ belongs to at least one $5^+$-face, and so $v$ gains a charge of at least $\frac{3}{4} + \frac{3}{4} + \frac{1}{2} = 2$.
\end{proof}

\begin{claim}
    After discharging, all $3^+$-vertices have non-negative charge.
\end{claim}

\begin{proof}
    There are no $3$-vertices in $G$ by \Cref{claim:3vertex}, and $4$-vertices do not lose any charge as they are not adjacent to any $2$-vertices by \Cref{claim:2-4}. By \Cref{claim:567}, vertices of degree $4+t$ lose at most $t$, for $t \in \{1,2,3\}$. Finally, if a vertex $v$ has degree $d \ge 8$, we claim that $v$ loses a charge of at most $d/2$. This is clear if all neighbors of $v$ are $2$-vertices, thus protected, therefore we shall assume that there is at least one $3^+$-vertex in $N(v)$. A \emph{block} around $v$ consists of a $3^+$-vertex and all of its consecutive successors of degree $2$ until the next $3^+$-vertex in $\sigma_v$. Consider any block $u_1, \dots, u_k$: we claim that this block takes a charge of at most $k/2$ from $v$. If $k=1$, then no charge is taken from $v$. If $k=2$, the block contains an isolated vertex that receives a charge of $1$ from $v$. If $k\ge3$, the block contains two fringe neighbors and $k-3$ protected neighbors, and thus receives a total charge of $3/2 + (k-3)/2 = k/2$ from $v$. Since all blocks around $v$ cost at most half their size, and form a disjoint union of $N(v)$, we conclude that $v$ gives a charge of at most $d/2$ to its neighborhood.
\end{proof}

\subsection{Upper bounds for girth 5 and 6}

The goal of this section is to show the following.

\begin{theorem}
\label{thm:girth56}
    Every planar graph of girth 5 has inversion diameter at most $7$,
    and every planar graph of girth 6 has inversion diameter at most $5$. 
\end{theorem}

To prove this, we combine Proposition~\ref{prop:hom} with the results from Montejano et al.~\cite{MONTEJANO20101365}. These provide targets for planar graphs of girth $5$ (or $6$).
We may then show that these targets have small inversion diameter to conclude. 

Let us first define the considered targets. Given a prime power $q =1\mod 4$, the \emph{Paley graph} of order $q$, denoted by $SP_q$, is a clique on $q$ vertices, each of them associated with one element of the finite field $\mathbb{F}_q$, and each edge $xy$ is colored $1$ if $x-y$ is a square, $0$ otherwise. It is well-known that exchanging $0$ and $1$ in the latter definition gives an isomorphic 2-edge-colored graph: we say that $SP_q$ is \emph{autodual}.

Given a $2$-edge-colored graph $(G,\pi)$, its \emph{Tromp} graph $(\hat{G},\hat{\pi})$ is formed from two copies of $G$ together with two additional vertices as follows:
\begin{itemize}
    \item $V(\hat{G})=V(G)\cup\{u'; u\in V(G)\}\cup\{\infty,\infty'\}$
    \item $E(\hat{G}) = \bigcup_{u \in V(G)} \{u\infty, u\infty', u'\infty, u'\infty'\}\cup \bigcup_{uv \in E(G)}\{uv, uv', u'v, u'v'\}$
    \item For each $u\in V(G)$, $\hat{\pi}(u\infty)=\hat{\pi}(u'\infty')=0$ and $\hat{\pi}(u\infty')=\hat{\pi}(u'\infty)=1$.
    \item For each edge $uv\in E(G)$, $\hat{\pi}(uv)=\hat{\pi}(u'v')=\pi(uv)$ and $\hat{\pi}(u'v)=\hat{\pi}(uv')=1-\pi(uv)$.
\end{itemize}
Note that each vertex $u'$ is an anti-twin of $u$, meaning they have the same neighborhood, but with opposite colors: $uv$ and $u'v$ always get distinct colors.

With these definitions, we may now recall the results from~\cite{MONTEJANO20101365}.

\begin{theorem}[\cite{MONTEJANO20101365}]
$\hat{SP_9}$ (resp. $\hat{SP_5}$) is a target for planar graphs of girth $5$ (resp. $6$).
\end{theorem}

We can now get our upper bounds by computing the inversion diameter of $\hat{SP_9}$ and $\hat{SP_5}$. We first show how the inversion diameter is impacted by the Tromp operation.

\begin{proposition}
    For every 2-edge-colored graph $(G,\pi)$, letting $(\hat{G},\hat{\pi})$ be its Tromp graph, $\|\hat{\pi}\|\leq \|1-\pi\|+2$.
\end{proposition}

\begin{proof}
    Consider an inversion $\varphi$ of dimension $k:=\|1-\pi\|$ for $(G,1-\pi)$. We construct an inversion $\hat{\varphi}$ of dimension $k+2$ for $(\hat{G},\hat{\pi})$ as follows:
    \begin{itemize}
        \item For each $u\in V(G)$, $\hat{\varphi}(u)=(0,1,\varphi(u)_1,\ldots,\varphi(u)_k)$ and $\hat{\varphi}(u')=(1,0,\varphi(u)_1,\ldots,\varphi(u)_k)$. 
        \item $\hat{\varphi}(\infty)=(1,0,0,\ldots,0)$, $\hat{\varphi}(\infty')=(0,1,0,\ldots,0)$.
    \end{itemize}
    Now observe that if $u\neq v\in V(G)$, 
    \begin{itemize}
        \item $\hat{\varphi}(u)\cdot\hat{\varphi}(\infty)=0=\hat{\pi}(u\infty)$.
        \item $\hat{\varphi}(u')\cdot\hat{\varphi}(\infty)=1=\hat{\pi}(u'\infty)$.
        \item $\hat{\varphi}(u)\cdot\hat{\varphi}(v)=1+\varphi(u)\cdot\varphi(v)=1+(1-\pi)(uv)=\hat{\pi}(uv)$.
        \item $\hat{\varphi}(u)\cdot\hat{\varphi}(v')=\varphi(u)\cdot\varphi(v)=(1-\pi)(uv) =\hat{\pi}(uv')$.
    \end{itemize}
    The remaining cases being symmetric, we get that $\hat{\varphi}$ is an $k+2$-inversion for $(\hat{G},\hat{\pi})$, which concludes.
\end{proof}

Recalling that each of the targets $SP_9$ and $SP_5$ are autodual, the previous proposition allows us to conclude by providing a $5$-inversion for $SP_9$ and a $3$-inversion for $SP_5$. Those are provided in \Cref{SZ4} (where for simplicity only the 1-colored edges are drawn), which ends the proof of Theorem~\ref{thm:girth56}.

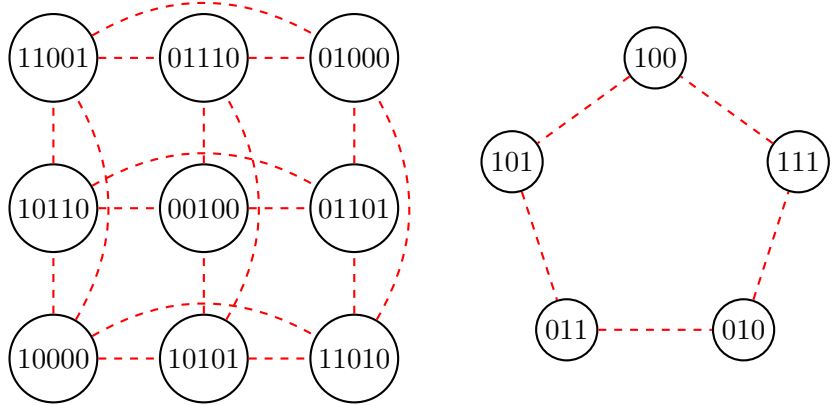
\begin{figure}[!ht]
    \centering
    \begin{tikzpicture}[thick, every node/.style={draw,circle,inner sep=2pt},scale=2]
    \node (00) at (0,0) {10000};
    \node (01) at (1,0) {10101};
    \node (02) at (2,0) {11010};
    \node (10) at (0,1) {10110};
    \node (11) at (1,1) {00100};
    \node (12) at (2,1) {01101};
    \node (20) at (0,2) {11001};
    \node (21) at (1,2) {01110};
    \node (22) at (2,2) {01000};
    
    \draw[red, dashed] (00) -- (01) -- (02);
    \draw[red, dashed] (10) -- (11) -- (12);
    \draw[red, dashed] (20) -- (21) -- (22);
    \draw[red, dashed] (00) -- (10) -- (20);
    \draw[red, dashed] (01) -- (11) -- (21);
    \draw[red, dashed] (02) -- (12) -- (22);
    \draw[bend left,red, dashed] (00) to (02);
    \draw[bend left,red, dashed] (10) to (12);
    \draw[bend left,red, dashed] (20) to (22);
    \draw[bend right,red, dashed] (00) to (20);
    \draw[bend right,red, dashed] (01) to (21);
    \draw[bend right,red, dashed] (02) to (22);

    \tikzset{xshift=4cm,yshift=1cm}
    \node (0) at (90:1) {100};
    \node (1) at (162:1) {101};
    \node (2) at (234:1) {011};
    \node (3) at (306:1) {010};
    \node (4) at (18:1) {111};
    \draw[red, dashed] (0) -- (1) -- (2) -- (3) -- (4) -- (0) ;
    \end{tikzpicture}
    \caption{Valid vector assignments for $SP_9$ (left) and $SP_5$ (right) of optimal dimensions.}
    \label{SZ4}
\end{figure}

\subsection{Upper bounds for girth 7}

We conclude this section by providing our tight upper bound on the inversion diameter of planar graphs with girth at least $7$. A $t$-inversion is \emph{strict} if every vertex is associated with a non-null vector. We show the following:
\begin{theorem}
\label{thm:girth7}
    Let $G$ be a planar graph of girth at least $7$ and $\pi:E(G)\to\mathbb{F}_2$. There exists a strict $3$-inversion of $(G,\pi)$.
\end{theorem}

The proof uses a mild discharging argument. We assume Theorem~\ref{thm:girth7} has a counter-example $(G,\pi)$, and choose one with minimum number of edges. We make use of the results from~\cite[Section 8.2]{havet} that we summarize below. 

\begin{proposition}[\cite{havet}]
\label{prop:havetdisch}
$G$ cannot contain any of the following: 
\begin{itemize}
    \item A vertex of degree $1$.
    \item Two adjacent vertices of degree $2$.
    \item A vertex of degree at most $6$ all of whose neighbors have degree $2$. 
    \item A vertex of degree $3$ with two neighbors of degree $2$.
\end{itemize}
\end{proposition}

A \emph{deficient} vertex is a vertex of degree $3$ with a (unique) neighbor of degree $2$. We add three new configurations, and show that $G$ cannot contain them in the three upcoming lemmas. All of them rely on the following claims, which can be checked easily by a dimension argument or by exhaustive search.

\begin{claim}
\label{cl:claim_disch}
    Given two distinct non-zero vectors $\mathbf{x},\mathbf{y}\in\mathbb{F}_2^3$, and any $b\in\{0,1\}$, there are at least $6$ vectors $\mathbf{z}$ satisfying $\mathbf{x}\cdot \mathbf{z} =b$ or $\mathbf{y}\cdot \mathbf{z}=b$.
\end{claim}

\begin{claim}
\label{cl:deg2}
Given two non-zero vectors $\mathbf{x},\mathbf{y}\in\mathbb{F}_2^3$ and $a,b,c\in\{0,1\}$, there are two vectors $\mathbf{z}\in\mathbb{F}_2^3\setminus\{\mathbf{0}\}$ such that $\mathbf{z}\cdot \mathbf{x} = a$ and $\mathbf{z}$ admits a vector $\mathbf{z'}\in\mathbb{F}_2^3\setminus\{\mathbf{0}\}$ with $\mathbf{z'}\cdot \mathbf{z}=b$ and $\mathbf{z'}\cdot \mathbf{y}=c$.
\end{claim}

\begin{proof}
    Observe that there are four possible vectors $\mathbf{z}$ such that $\mathbf{z}\cdot \mathbf{x}=a$, and at least two of them must be different from $\mathbf{0}$ and $\mathbf{y}$. We then find $\mathbf{z'}$ using the argument from Lemma 8.13 in~\cite{havet} that we recall hereafter.

    Since $\mathbf{y}$ and $\mathbf{z}$ are distinct and non-zero, they are linearly independent, hence the set of solutions of the affine equations $\mathbf{z'}\cdot \mathbf{z} = b$ and $\mathbf{z'}\cdot \mathbf{y}=c$ is a non-empty affine space of dimension $1$, and therefore contains a non-zero vector.
\end{proof}

We may now show that $G$ cannot contain some configurations.
\begin{lemma}
\label{lem:nondef}
    Every deficient vertex of $G$ has a non-deficient neighbor of degree at least $3$.
\end{lemma}

\begin{proof}
    Assume otherwise, and let $v$ be a deficient vertex with two deficient neighbors $u,w$ in $G$. Denote by $u',v',w'$ the neighbors of degree $2$ of $u,v,w$, and by $a,b,c,d,e$ the remaining vertices as shown in Figure~\ref{fig:nondef}. Note that the girth assumption ensures that these vertices are all distinct and that all edges between $\{u,v,w,u',v',w'\}$ are drawn in the figure.

    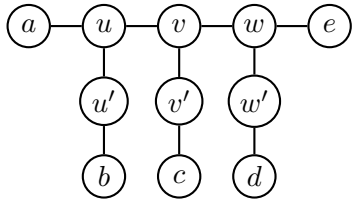
\begin{figure}[!ht]
    \centering
        \begin{tikzpicture}[thick, every node/.style={draw,circle,inner sep=2pt,minimum size=16pt}]
        \node (a) at (0,0) {$a$};
        \node (b) at (1,-2) {$b$};
        \node (c) at (2,-2) {$c$};
        \node (d) at (3,-2) {$d$};
        \node (e) at (4,0) {$e$};
        \node (u) at (1,0) {$u$};
        \node (v) at (2,0) {$v$};
        \node (w) at (3,0) {$w$};
        \node (u') at (1,-1) {$u'$};
        \node (v') at (2,-1) {$v'$};
        \node (w') at (3,-1) {$w'$};
        \draw (a) -- (u) -- (v) -- (w) -- (e);
        \draw (b) -- (u') -- (u);
        \draw (c) -- (v') -- (v);
        \draw (d) -- (w') -- (w);
        \end{tikzpicture}
    \caption{The configuration of Lemma~\ref{lem:nondef}}
    \label{fig:nondef}
    \end{figure}

    Since $(G,\pi)$ is a minimal counter-example, there exists a strict $3$-inversion of $(G-\{u,v,w,u',v',w'\},\pi)$ that we denote in bold. Let $\mathbf{u_1},\mathbf{u_2}$ the choices for $\mathbf{u}$ provided by Claim~\ref{cl:deg2} (note that both these choices can be extended into $\mathbf{u'}$ in a way that is compatible with the vectors assigned on $G-\{u,v,w,u',v',w'\}$). We similarly get $\mathbf{w_1},\mathbf{w_2}$ for $\mathbf w$.
    
    Let $U$ (resp. $W$) the set of vectors for $v$ that are compatible with either $\mathbf{u_1}$ or $\mathbf{u_2}$ (resp. $\mathbf{w_1}$ or $\mathbf{w_2}$). By Claim~\ref{cl:claim_disch}, $|U|\geq 6$ and $|W|\geq 6$, hence $|U\cap W|\geq 4$. Now define $\mathbf{v_1}\in U\cap W\setminus\{\mathbf{0},\mathbf{c}\}$. By construction, one can assign $\mathbf{u_1}$ or $\mathbf{u_2}$ to $u$, $\mathbf{w_1}$ or $\mathbf{w_2}$ to $w$, $\mathbf{v_1}$ to $v$ and find vectors for $u',v',w'$ by Claim~\ref{cl:deg2} to get a strict $3$-inversion for $(G,\pi)$, a contradiction.
\end{proof}

\begin{lemma}
    \label{lem:deg3}
    $G$ cannot contain a vertex of degree $3$ with three deficient neighbors.
\end{lemma}

\begin{proof}
    Assume otherwise, and that $G$ contains the structure depicted in Figure~\ref{fig:deg3}. Note that the girth assumption ensures that these vertices are all distinct and that all edges of $G$ among them are drawn. Let $H$ be obtained from $G$ by removing $u,u_1,u_2,u_3,v_1,v_2,v_3$. By minimality, there exists a strict $3$-inversion of $(H,\pi|_{E(H)})$, denoted in bold. 
    \begin{figure}[!ht]
        \centering
        \begin{tikzpicture}[thick, every node/.style={draw,circle,inner sep=2pt,minimum size=16pt}]
        \node (u) at (0,0) {$u$};
        \node (u1) at (-1,-1) {$u_1$};
        \node (u2) at (0,-1) {$u_2$};
        \node (u3) at (1,-1) {$u_3$};
        \node (v1) at (-1,-2) {$v_1$};
        \node (v2) at (0,-2) {$v_2$};
        \node (v3) at (1,-2) {$v_3$};
        \draw (1,-2.5) -- (v3) -- (u3) -- (u) -- (u2) -- (v2) -- (0,-2.5);
        \draw (-1,-2.5) -- (v1) -- (u1) -- (u); 
        \draw (-1.5,-1.5) -- (u1);
        \draw (-.5,-1.5) -- (u2);
        \draw (.5,-1.5) -- (u3);
        \end{tikzpicture}
        \caption{The configuration of Lemma~\ref{lem:deg3}}
        \label{fig:deg3}
    \end{figure}
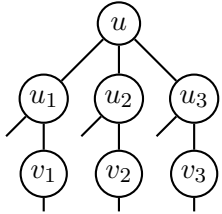

    For $i\in\{1,2,3\}$, we may apply Claim~\ref{cl:deg2} to find two choices for $\mathbf{u_i}$ that can be extended to $\mathbf{v_i}$, in a way that is compatible with the vectors in $H$. Let $U_i$ be the set of candidate vectors for $\mathbf{u}$ that are compatible with at least one choice of $\mathbf{u_i}$.

    By Claim~\ref{cl:claim_disch}, $|U_i|\geq 6$. There are $8$ vectors in $\mathbb{F}_2^3$, $U_1\cap U_2\cap U_3$ that have size at least $2$ and we can choose $\mathbf{u}$ to be a non-zero vector in this set. Therefore by definition of $U_i$ and Claim~\ref{cl:deg2}, we can assign $\mathbf{u}$ to $u$ and extend this into a strict $3$-inversion of $(G,\pi)$, a contradiction.
\end{proof}

\begin{lemma}
\label{lem:deg4}
    $G$ cannot contain a vertex of degree $4$ with $3$ neighbors of degree $2$ and a deficient neighbor.
\end{lemma}

\begin{proof}
Assume otherwise, and that $G$ contains the structure depicted on Figure~\ref{fig:deg4}. Let $H$ be obtained from $G$ be removing $u,v,u_1,u_2,u_3,v_1$. By minimality, there exists a strict $3$-inversion of $(H,\pi|_{E(H)})$, denoted in bold. 
    \begin{figure}[!ht]
        \centering
        \begin{tikzpicture}[thick, every node/.style={draw,circle,inner sep=2pt,minimum size=16pt}]
            \node (u) at (0,0) {$u$};
            \node (u1) at (-1,0) {$u_1$};
            \node (a) at (-2,0) {$a$};
            \node (u2) at (0,1) {$u_2$};
            \node (b) at (0,2) {$b$};
            \node (u3) at (1,0) {$u_3$};
            \node (c) at (2,0) {$c$};
            \node (v) at (0,-1) {$v$};
            \node (e) at (-.5,-2) {$e$};
            \node (v1) at (.5,-2) {$v_1$};
            \node (d) at (1.5,-2) {$d$};
            \draw (d) -- (v1) -- (v) -- (u) -- (u3) -- (c);
            \draw (b) -- (u2) -- (u) -- (u1) -- (a);
            \draw (e) -- (v);
        \end{tikzpicture}
        \caption{The configuration of Lemma~\ref{lem:deg4}}
        \label{fig:deg4}
    \end{figure}
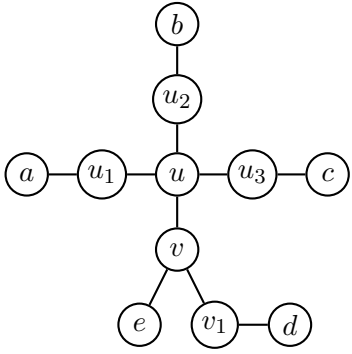

By Claim~\ref{cl:deg2}, there are two possible choices of $\mathbf{v}$ that can be extended into a choice of $\mathbf{v_1}$ which is compatible with $(H,\pi)$. By Claim~\ref{cl:claim_disch}, there are thus six choices for $\mathbf{u}$. Avoiding $\mathbf{0}$, $\mathbf{a}$, $\mathbf{b}$ and $\mathbf{c}$ yields a choice for $\mathbf{u}$, which may then be extended into a strict $3$-inversion of $(G,\pi)$, a contradiction.
\end{proof}

We now reach a contradiction using the following discharging argument. Fix a planar embedding of $G$. Assign to each vertex $v$ of $G$ a weight $\omega(v)=5\deg(v)-14$, and to each face $f$ of $G$ a weight $\omega(f)=2\deg(f)-14$. Observe that the total weight $\sum_v \omega(v)+\sum_f\omega(f)$ is $-28$ by Euler's formula. We now move the weights as follows: 
\begin{itemize}
\item Each vertex gives $2$ to each of its neighbors of degree $2$.
\item Each vertex of degree at least $4$ gives $1$ to each of its deficient neighbors. 
\item Each non-deficient vertex of degree $3$ gives $1/2$ to each of its deficient neighbors. 
\end{itemize}

We claim that after this process, every vertex and face of $G$ has non-negative weight.
This will yield a contradiction since the total initial weight was negative, thus concluding the proof of Theorem~\ref{thm:girth7}.

The case of faces and vertices of degree $2$ is straightforward. Vertices of degree $d\geqslant 5$ give a weight of at most $2d$, hence their final weight is at least $3d-14\geqslant 1$.

\begin{lemma}
 Vertices of degree $4$ have non-negative final weight.
\end{lemma}
\begin{proof}
    Let $v$ be a vertex of degree $4$, whose initial weight is $6$. It cannot have four neighbors of degree $2$ by Proposition~\ref{prop:havetdisch}. Moreover, if it has three, the remaining vertex cannot be deficient by Lemma~\ref{lem:deg4}, so $v$ gives exactly $6$ weight. Otherwise, $v$ has at most $2$ neighbors of degree $2$, hence loses at most $2\times 2 + 2\times 1=6$, ending with non-negative weight.
\end{proof}

\begin{lemma}
     Vertices of degree $3$ have non-negative final weight.
\end{lemma}

\begin{proof}
    Let $v$ be a vertex of degree $3$, whose initial weight is $1$. If $v$ is not deficient, then it has at most two deficient neighbors by Lemma~\ref{lem:deg3}, hence gives $2\times 1/2=1$ and ends up with $0$ weight.

    Assume now that $v$ is deficient, and denote by $w_1,w_2$ its two neighbors of degree at least $3$. If $w_1,w_2$ both have degree $3$, then they each give $1/2$ to $v$, which ends up with weight $1-2+2\times 1/2=0$. Otherwise, $w_1$ or $w_2$ has degree at least $4$, and gives $1$ to $v$, for a final weight of at least $1-2+1=0$.
\end{proof}

\section{Maximum degree bound for triangle-free graphs}\label{sec:triangle-free}

This section is devoted to the proof of \Cref{thm:main}, bounding $\diam$ in terms of the maximum degree $\Delta$ for triangle-free graphs.
Our proof relies on the linear algebra formulation of $\diam$ given by \Cref{obs:main}.
We will first need the following simple lemma.
\begin{lemma}\label{lem:main-aux}
    Let $\mathbf{x}\in\mathbb{F}_2^t$ and $\mathbf{y}+E$ be an affine subspace of $\mathbb{F}_2^t$. 

    If $\mathbf{x}$ is not orthogonal to $E$, then one can find two vectors $\mathbf{y_0},\mathbf{y_1}\in\mathbf{y}+E$ such that $\mathbf{x\cdot y_i}=i$ for $i=0,1$. 
\end{lemma}
\begin{proof}
    Since $\mathbf{x}$ is not orthogonal to $E$, there exists $\mathbf{x'}\in E$ such that $\mathbf{x\cdot x'}=1$. In particular, the dot products $\mathbf{x\cdot y}$ and $\mathbf{x}\cdot (\mathbf{x'}+\mathbf{y})$ have different values, which concludes.  
\end{proof}

We are now equipped to show our bound on $\diam$ in terms of the maximum degree $\Delta$.

\thmtrianglefree*

\begin{proof}
Let $G$ be a triangle-free graph of maximum degree $\Delta$. Let $t=\Delta+\lfloor\log\Delta\rfloor$. Using Observation~\ref{obs:main}, we prove Theorem~\ref{thm:main} by inductively constructing a $t$-inversion for every 2-edge-coloring $\pi$ of $G$ from a $t$-inversion of a graph with fewer vertices. Assume that Theorem~\ref{thm:main} is verified by all graphs with fewer vertices than $G$ (the base case being trivial). Let $\pi: E(G)\to\mathbb{F}_2$, and $u$ be a vertex of $G$, with neighbors $v_1,\ldots, v_k$ ($k\leq\Delta$). 

By hypothesis, $G\setminus\{u\}$ satisfies Theorem~\ref{thm:main}, hence one can find a valid $t$-inversion for $\pi|_{E(G\setminus\{u\})}$. Our plan is to modify the vectors $\mathbf{v_1},\ldots,\mathbf{v_k}$ so that we can find a vector $\mathbf{u}$ for $u$ satisfying the constraints given by the edges incident with $u$. 

Denote by $V_i$ the set of possible choices for $\mathbf{v_i}$ in $G\setminus\{u\}$, having decided all vector assignments in $G\setminus N[u]$. Observe that $V_i$ is non-empty, and is precisely the set of solutions of an affine system of (at most) $\Delta-1$ equations on $t$ unknowns (and this system does not depend on any $\mathbf{v_j}$ for $j\neq i$ since $G$ is triangle-free). By the Rouché-Capelli theorem, $V_i$ is an affine space of the form $\mathbf{v_i}+E_i$, where $E_i$ is a vector space of dimension at least $t-\Delta+1$.

According to~\Cref{lem:main-aux}, if we fix a vector $\mathbf{u}$ not orthogonal to $E_i$, then we can find $\mathbf{v'_i}\in\mathbf{v_i}+E_i$ such that $\mathbf{u\cdot v'_i}=\pi(uv_i)$. In particular, if we choose $\mathbf{u}$ outside of the union of the orthogonals of each $E_i$, we will be able afterwards to replace for $i\in [k]$, $\mathbf{v_i}$ by some $\mathbf{v'_i}\in V_i$ such that $f(uv_i)=\mathbf{u\cdot v'_i}$. This new assignment would then be a valid assignment for $G$, which will conclude the proof. 

Now, as the orthogonal ${E_i}^\perp$ of $E_i$ has dimension at most $\Delta-1$ for every $1\le i\le k$, the union $\bigcup_{i=1}^k {E_i}^\perp$ contains at most $k\cdot 2^{\Delta-1}\le 2^{\Delta+\log\Delta-1}<2^t$ vectors. In particular, we may choose $\mathbf{u}$ outside of this set, as required.
\end{proof}

\section{Subdivisions}\label{sec:subd}

We say that $H$ is a \emph{subdivision} of $G$ if it can be obtained from $G$ by replacing every edge by a path on at least 2 vertices (the lengths of the paths may be different for different edges). The goal of this section is to prove Theorem~\ref{thm:subd}, which we restate below.

\thmsubd*

\begin{proof}
Let $H$ be a subdivision of $G$, such that $V(G) \subseteq V(H)$, and each edge $uv \in E(G)$ is replaced by a path of length at least one. For each such path $(u,\ldots,v)$ in $H$, we let $u'$ be the neighbor of $u$ and $v'$ be the neighbor of $v$ (possibly with $u'=v,v'=u$ when $uv$ is not subdivided, and $v'=u'$ when it is subdivided once).
We use the point of view given by Observation~\ref{obs:norm}, by considering an arbitrary $2$-edge-coloring $\pi$ of $H$, and showing that $\|\pi\| = \diam(H)$ is bounded. 

Let us say that a 2-edge-coloring of $H$ is \emph{good} if
 for each $uv \in E(G)$, $uu'$ and $v'v$ are of the same color.
The following claim shows how to obtain a \emph{good} 2-edge-coloring from $\pi$.
\begin{claim}
    \label{claim:subd1}
    With at most $\log(\diam(G))+3$ inversions, $\pi$ can be transformed into a good 2-edge-coloring $\pi'$ of $H$. 
\end{claim}

\begin{proofof}{the claim}
    Letting $\nabla_0$ for the greatest average degree of $G$, a standard degeneracy argument ensures that $G$ is $(\nabla_0+1)$-colorable.
    It was shown in~\cite{havet} (Theorem 3.2) that $\nabla_0\leq 2\diam(G)$, meaning we may have a proper vertex-coloring of $G$ with $k \leq 2 \diam (G) + 1$ colors, and in particular $k \leq 4 \diam(G)$. 
    We take such a coloring $\phi$, labeling the $k$ colors with distinct vectors of $\mathbb{F}_2^\ell$, with $\ell = \lceil \log(k) \rceil \leq \log(k)+1 \leq \log(\diam(G))+3$. We are now ready to define the sets  $X_1,...,X_{\ell}$ for the $\ell$ inversions.

    For any $i\in[\ell]$, we begin by including in $X_i$ all vertices $v\in V(G)$ such that the $i$-th coordinate of $\phi(v)$ is 1.
    Then, for any edge $(u,v) \in E(G)$ such that $\pi(uu') \neq \pi(vv')$, we choose a (single) index $i \leq \ell$ such that exactly one of $u$ and $v$ belongs to $X_i$ (which is possible since $\phi(u) \neq \phi(v)$).
    Then, we add $u'$ (resp. $v'$) to $X_i$ if $u\in X_i$ (resp. $v\in X_i$).
    Applying inversions $(X_i)_{i\in [\ell]}$, observe that for every edge $uv$ of $G$, either $\pi(uu')=\pi(vv')$ and $uu',vv'$ are never inverted, or exactly one edge among $uu',vv'$ gets recolored (exactly once), which yields a good edge-coloring.
\end{proofof}

We now consider $(H,\pi')$ obtained from the previous claim.
Then, consider the subgraph of $H$ induced by the subdivision vertices ($V(H) \setminus V(G)$), and note that it induces a (path) forest, and therefore has inversion diameter at most $2$ by~\cite[Theorem 4.2]{havet}. We may thus apply at most two inversions to $(H, \pi')$ to yield some $(H,\pi'')$, where each subdivision path is monochromatic, meaning $\|\pi'\| \leq \| \pi''\| + 2$, and in turn $\| \pi \| \leq \|\pi''\| +  \log(\diam(G)) + 5$.

Now, let $\hat{\pi}$ be the 2-edge-coloring of $G$ associating to each $uv \in E(G)$ the color of $(u,u',\ldots,v',v)$ in $(H,\pi'')$.
Then, we claim $\| \pi'' \| \leq \|\hat{\pi}\|$, because any inversion sequence $(X_i)_{i\in[\ell]}$ of $(G,\hat{\pi})$ ending in a $0$-coloring of $G$ may be extended to $(H,\pi'')$. Indeed, it suffices to augment the sets $(X_i)_{i\in[\ell]}$ as follows: for every $i$, and every $uv \in E(G)$ such that $\{u,v \} \in X_i$, add all subdivided vertices of $uv$ to $X_i$. Then, at step $i$, the augmented sequence on $H$ yields exactly the $2$-edge-colored subdivision of $G$ at step $i$.
We thus have $\|\pi''\| \leq \diam(G)$, which combined with our previous bound on $\pi$ yields $\|\pi\| \leq \diam(G) + \log(\diam(G)) + 5$, concluding the proof.
\end{proof}

We note that if $H$ is obtained from $G$ by subdividing each edge at least once, a similar strategy yields a bound of $\log(\diam(G))+O(1)$, by short-circuiting the last part of the proof of \Cref{thm:subd}.

\bibliographystyle{plain}

\bibliography{biblio}
\end{document}